\documentclass[11pt,a4paper]{amsart}

\usepackage[T1]{fontenc}
\usepackage{amsmath,amssymb,amsthm}
\usepackage{booktabs}
\usepackage{geometry}
\usepackage[
  hidelinks,
  pdftitle={An Explicit Counterexample to Stanley’s Rankwise Lower-Bound Conjecture for Differential Posets},
  pdfauthor={Xinan Dai, Wenhao Deng, Yingdong Shi, Tailin Wu, and Yuchen Yang},
  pdfsubject={Differential posets and extremal rank sizes},
  pdfkeywords={differential poset, Young's lattice, rank function, extremal poset, incidence trade}
]{hyperref}

\newtheorem{theorem}{Theorem}[section]
\newtheorem{lemma}[theorem]{Lemma}
\newtheorem{proposition}[theorem]{Proposition}
\theoremstyle{definition}
\newtheorem{definition}[theorem]{Definition}
\theoremstyle{remark}
\newtheorem{remark}[theorem]{Remark}

\newcommand{\card}[1]{\lvert #1\rvert}
\newcommand{\Rcal}{\mathcal{R}}
\newcommand{\Acal}{\mathcal{A}}

\title[Stanley's Differential-Poset Rank-Minimum Problem]
{An Explicit Counterexample to Stanley's Rankwise\\ 
Lower-Bound Conjecture for Differential Posets}

\author{Xinan Dai}
\thanks{Author note: Xinan Dai received his B.S. degree in Mathematics
and Applied Mathematics from Shanghai University. He is currently
enrolled as a Ph.D. student at Fudan University and is a visiting
student at the AI for Scientific Simulation and Discovery Lab,
Westlake University.}

\address{
Key Laboratory for Information Science of Electromagnetic Waves,
College of Future Information and Technology,
Fudan University, Shanghai, China
}

\curraddr{Department of Artificial Intelligence,
School of Engineering, Westlake University, Hangzhou, China}
\email{xndai23@m.fudan.edu.cn}

\author{Wenhao Deng}
\thanks{Wenhao Deng is a student at the University of Glasgow and is currently an intern at the AI for Scientific Simulation and Discovery Lab, Westlake University.}
\address{University of Glasgow, Glasgow, United Kingdom}
\curraddr{Department of Artificial Intelligence,
School of Engineering, Westlake University, Hangzhou, China}
\email{dengwenhao@westlake.edu.cn}

\author{Yingdong Shi}
\address{School of Information Science and Technology, ShanghaiTech University, Shanghai, China}
\email{shiyd2023@shanghaitech.edu.cn}

\author{Tailin Wu}
\address{Department of Artificial Intelligence,
School of Engineering, Westlake University, Hangzhou, China}
\email{wutailin@westlake.edu.cn}

\author{Yuchen Yang}
\address{Department of Artificial Intelligence,
School of Engineering, Westlake University, Hangzhou, China}
\email{yangyuchen@westlake.edu.cn}

\date{24 July 2026}


\begin{document}

\begin{abstract}
In Problem~6 of his 1988 paper on differential posets, Stanley asked for the least possible cardinality of a fixed rank of an $r$-differential poset and suggested that the minimum should be attained by $Y^r$, the $r$-fold Cartesian power of Young's lattice. We disprove the resulting universal coefficientwise lower bound. For every $r\geq 3$, we construct an infinite $r$-differential poset $P^{(r)}$ satisfying
\[
  \card{P^{(r)}_4}
  =\card{(Y^r)_4}-\left\lfloor\frac r3\right\rfloor.
\]
For $r=3$, the construction replaces thirteen rank-four lower-cover blocks of $Y^3$ by twelve blocks with the same point and pair incidence multiplicities, producing the initial rank sequence $1,3,9,22,50$ instead of $1,3,9,22,51$. A reflection extension then yields an infinite differential poset. The construction does not address the cases $r=1$ and $r=2$.
\end{abstract}

\maketitle

\section*{Introduction}

Stanley introduced differential posets as an axiomatic framework retaining many local and enumerative features of Young's lattice~\cite{Stanley1988}. In Problem~6 of his original paper, he asked for the greatest and least possible number of elements of rank $n$ in an $r$-differential poset and suggested that $Y^r$ should be the rank-minimizing example. Byrnes subsequently settled the maximum side of the problem~\cite{Byrnes2012}. Stanley and Zanello later restated the proposed lower bound explicitly~\cite[Question~17]{StanleyZanello2012}: must every $r$-differential poset $P$ satisfy
\begin{equation}\label{eq:stanley-bound}
  \card{P_n}\geq \card{(Y^r)_n}
  \qquad(n\geq0)?
\end{equation}

We show that the universally quantified assertion in \eqref{eq:stanley-bound} is false. The basic counterexample occurs at $r=3$ and rank four. Its finite mechanism is an unequal replacement of lower-cover blocks: thirteen blocks in $Y^3$ are replaced by twelve new blocks while preserving every one-point and two-point incidence multiplicity. These multiplicities are precisely the entries of the operator $D_4U_3$. Consequently, the differential-poset relation remains valid through rank three, although the fourth rank has one fewer element. Stanley's reflection construction then extends the finite initial segment to an infinite differential poset.

Section~\ref{sec:preliminaries} recalls the relevant definitions and states the main result. Section~\ref{sec:r3-construction} gives the rank-four replacement for $r=3$ and verifies it directly. Section~\ref{sec:extension} gives the infinite extension, and Section~\ref{sec:general-r} transports the construction to every $r\geq3$.

\section{Differential posets and the main result}\label{sec:preliminaries}

Young's lattice $Y$ is the graded poset of integer partitions ordered by inclusion of Ferrers diagrams. Its rank-$n$ elements are the partitions of $n$, and a cover relation adds one cell. Write $p(n)=\card{Y_n}$ for the number of integer partitions of $n$. Hence
\[
  F_Y(q)=\sum_{n\geq0}\card{Y_n}q^n
  =\prod_{i\geq1}(1-q^i)^{-1}.
\]
The Cartesian product of $r$ copies is denoted by $Y^r$. An element of $(Y^r)_n$ is an $r$-tuple of partitions whose total size is $n$, and
\[
  F_{Y^r}(q)=F_Y(q)^r.
\]

\begin{definition}\label{def:differential-poset}
Let $r$ be a positive integer. An \emph{$r$-differential poset} is a locally finite graded poset
\[
  P=\bigsqcup_{n\geq0}P_n
\]
with a unique least element and the following properties.
\begin{enumerate}
  \item[\textup{(D1)}] If $x\neq y$ have the same rank, then the number of their common upper covers equals the number of their common lower covers.
  \item[\textup{(D2)}] If $x$ has $d$ lower covers, then it has $d+r$ upper covers.
\end{enumerate}
\end{definition}

Let $\mathbb{Q}P_n$ be the vector space with basis $P_n$. Define the up and down operators by
\[
  U_nx=\sum_{\substack{z\in P_{n+1}\\z\gtrdot x}}z,
  \qquad
  D_nx=\sum_{\substack{w\in P_{n-1}\\x\gtrdot w}}w.
\]
The two axioms are equivalent to
\begin{equation}\label{eq:differential-identity}
  D_{n+1}U_n-U_{n-1}D_n=rI_{\mathbb{Q}P_n}
  \qquad(n\geq0),
\end{equation}
with the evident convention in rank zero. Young's lattice is $1$-differential, and Cartesian products add the parameters; hence $Y^r$ is $r$-differential.

\begin{theorem}\label{thm:main}
For every integer $r\geq3$, there exists an infinite $r$-differential poset $P^{(r)}$ such that
\[
  \card{P^{(r)}_i}=\card{(Y^r)_i}\quad(0\leq i\leq3),
  \qquad
  \card{P^{(r)}_4}=\card{(Y^r)_4}-\left\lfloor\frac r3\right\rfloor.
\]
Moreover,
\begin{equation}\label{eq:Y-r-rank4}
  \card{(Y^r)_4}=\frac{r(r+1)(r^2+17r+42)}{24}.
\end{equation}
In particular, the first five rank sizes of $P^{(3)}$ are
\[
  (1,3,9,22,50),
\]
whereas those of $Y^3$ are $(1,3,9,22,51)$.
\end{theorem}

The rank-growth results of Miller~\cite{Miller2013} and Gaetz--Venkataramana~\cite{GaetzVenkataramana2020} do not imply the coefficientwise comparison in \eqref{eq:stanley-bound} and are consistent with Theorem~\ref{thm:main}.

\section{The finite \texorpdfstring{$r=3$}{r=3} construction}\label{sec:r3-construction}

\subsection{Lower-cover blocks}

Retain ranks zero through three of $Y^3$. For a proposed rank-four element $z$, set
\[
  B_z=\{x\in(Y^3)_3:x\lessdot z\}.
\]
We call $B_z$ the \emph{lower-cover block} of $z$.

\begin{lemma}[Point-and-pair criterion]\label{lem:point-pair}
Fix the poset through rank three. Let $(B_z)_{z\in Q}$ be a finite indexed family of nonempty subsets of $P_3$, one for each proposed rank-four element $z\in Q$. For $x,y\in P_3$, the $(x,y)$-entry of $D_4U_3$ is
\[
  \begin{cases}
    \card{\{z\in Q:x\in B_z\}},&x=y,\\[2mm]
    \card{\{z\in Q:\{x,y\}\subseteq B_z\}},&x\neq y.
  \end{cases}
\]
Consequently, replacing rank-four blocks by blocks having the same incidence multiplicity on every one- and two-element subset of $P_3$ leaves $D_4U_3$ unchanged.
\end{lemma}

\begin{proof}
For $x=y$, the coefficient of $x$ in $D_4U_3x$ is the number of rank-four elements covering $x$. For $x\neq y$, the coefficient of $y$ is the number of rank-four elements covering both $x$ and $y$.
\end{proof}

Because the ranks below four are unchanged, $U_2D_3$ is unchanged as well. It therefore suffices to preserve the point and pair incidence multiplicities of the rank-four lower-cover blocks.

\subsection{Indexing rank three}

Write a triple of partitions as $(\lambda;\mu;\nu)$, write $0$ for the empty partition, and concatenate the parts of a partition; for example, $21$ denotes $(2,1)$. We label the elements of $(Y^3)_3$ as follows.

\begin{center}
\renewcommand{\arraystretch}{1.08}
\begin{tabular}{@{}r l r l@{}}
\toprule
Index & Element & Index & Element\\
\midrule
0  & $(3;0;0)$   & 11 & $(1;0;11)$\\
1  & $(21;0;0)$  & 12 & $(0;3;0)$\\
2  & $(111;0;0)$ & 13 & $(0;21;0)$\\
3  & $(2;1;0)$   & 14 & $(0;111;0)$\\
4  & $(2;0;1)$   & 15 & $(0;2;1)$\\
5  & $(11;1;0)$  & 16 & $(0;11;1)$\\
6  & $(11;0;1)$  & 17 & $(0;1;2)$\\
7  & $(1;2;0)$   & 18 & $(0;1;11)$\\
8  & $(1;11;0)$  & 19 & $(0;0;3)$\\
9  & $(1;1;1)$   & 20 & $(0;0;21)$\\
10 & $(1;0;2)$   & 21 & $(0;0;111)$\\
\bottomrule
\end{tabular}
\end{center}

\subsection{The thirteen-for-twelve replacement}\label{subsec:replacement}

Begin with the $51$ lower-cover blocks supplied by rank four of $Y^3$. Delete
\begin{equation}\label{eq:R-blocks}
\begin{aligned}
\Rcal=\{&\{1\},\{13\},\{20\},\{5,7\},\{5,8\},\{6,10\},\{6,11\},\\
&\{1,3,5\},\{1,4,6\},\{3,4,9\},\{5,6,9\},\{7,8,13\},\{10,11,20\}\}.
\end{aligned}
\end{equation}
In the displayed order, these are the lower-cover blocks of
\begin{equation}\label{eq:R-elements}
\begin{aligned}
 &(22;0;0),\ (0;22;0),\ (0;0;22),\ (11;2;0),\ (11;11;0),\ (11;0;2),\ (11;0;11),\\
 &(21;1;0),\ (21;0;1),\ (2;1;1),\ (11;1;1),\ (1;21;0),\ (1;0;21).
\end{aligned}
\end{equation}
The list in \eqref{eq:R-elements} is checked directly by deleting one removable corner cell from one of the three partition coordinates.
Insert twelve new rank-four elements with lower-cover blocks
\begin{equation}\label{eq:A-blocks}
\begin{aligned}
\Acal=\{&\{1,5\},\{1,6\},\{5,6\},\{7,13\},\{8,13\},\{10,20\},\{11,20\},\\
&\{1,3,4\},\{3,5,9\},\{4,6,9\},\{5,7,8\},\{6,10,11\}\}.
\end{aligned}
\end{equation}
Every block in \eqref{eq:A-blocks} is nonempty and therefore defines a rank-four element once the indicated cover relations are declared.

\subsection{Verification}

Only the twelve vertices
\[
  1,3,4,5,6,7,8,9,10,11,13,20
\]
occur in $\Rcal\cup\Acal$. Their incidence multiplicities agree:
\begin{equation}\label{eq:point-incidences}
\begin{array}{c|rrrrrrrrrrrr}
 v&1&3&4&5&6&7&8&9&10&11&13&20\\
\hline
 \deg_{\Rcal}(v)=\deg_{\Acal}(v)&3&2&2&4&4&2&2&2&2&2&2&2.
\end{array}
\end{equation}
Both sides contain each of the following twenty-two pairs exactly once:
\begin{equation}\label{eq:pair-incidences}
\begin{aligned}
&\{1,3\},\{1,4\},\{1,5\},\{1,6\},\{3,4\},\{3,5\},\{3,9\},\{4,6\},\{4,9\},\{5,6\},\\
&\{5,7\},\{5,8\},\{5,9\},\{6,9\},\{6,10\},\{6,11\},\{7,8\},\{7,13\},\{8,13\},\\
&\{10,11\},\{10,20\},\{11,20\},
\end{aligned}
\end{equation}
and neither side contains any other pair. Equations~\eqref{eq:point-incidences} and~\eqref{eq:pair-incidences} are therefore a complete incidence check.

Let $P^{(3)}_{\leq4}$ be obtained by leaving $(Y^3)_{\leq3}$ unchanged and replacing $\Rcal$ by $\Acal$ in rank four. By Lemma~\ref{lem:point-pair}, its $D_4U_3$ matrix is the same as that of $Y^3$, while $U_2D_3$ is unchanged. Hence
\[
  D_4U_3-U_2D_3=3I
\]
on rank three. The lower-rank identities are inherited from $Y^3$, so $P^{(3)}_{\leq4}$ is a partial $3$-differential poset and
\[
  \card{P^{(3)}_4}=51-13+12=50.
\]

\begin{remark}\label{rem:trade}
The replacement has a design-theoretic interpretation. Define a signed function on subsets of the twelve affected vertices by
\[
  f(B)=
  \begin{cases}
    +1,&B\in\Rcal,\\
    -1,&B\in\Acal,\\
    0,&\text{otherwise}.
  \end{cases}
\]
Then
\[
  \sum_{B\supseteq S}f(B)=0
  \qquad(1\leq\card{S}\leq2),
\]
whereas $\sum_Bf(B)=1$. After adjoining the empty block to the $\Acal$-side, the two collections form a simple $[2]$-trade of volume $13$ in the sense of~\cite{GhorbaniEtAl2020}. The differential-poset axioms detect the first two incidence moments but not the zeroth, which accounts for the one-element saving.
\end{remark}

\section{Extension to an infinite poset}\label{sec:extension}

The construction above is finite. The following reflection extension, appearing in~\cite[Proposition~6.1]{Stanley1988}, preserves its initial ranks.

\begin{proposition}[Reflection extension]\label{prop:reflection-extension}
Suppose $P_{\leq n}$ is a finite partial $r$-differential poset of rank $n$, meaning that \eqref{eq:differential-identity} holds in ranks below $n$. Then $P_{\leq n}$ extends by one rank to a partial $r$-differential poset of rank $n+1$. Iteration produces an infinite $r$-differential poset whose first $n+1$ ranks, including their cover relations, are exactly $P_{\leq n}$.
\end{proposition}

\begin{proof}
For each $y\in P_{n-1}$, add an element $y^*$ of rank $n+1$ that covers precisely the elements $x\in P_n$ covering $y$. For each $x\in P_n$, also add $r$ rank-$(n+1)$ elements, each covering only $x$.

If $x\in P_n$ has $d$ lower covers, then the elements $y^*$ provide $d$ upper covers and the singleton elements provide $r$ more. Thus $x$ has $d+r$ upper covers. If $x\neq x'$ lie in $P_n$, their common new upper covers are precisely the elements $y^*$ indexed by their common lower covers. The differential conditions therefore hold at rank $n$, and all lower-rank relations are unchanged. Iterating and taking the union gives the required infinite poset.
\end{proof}

Applying Proposition~\ref{prop:reflection-extension} to $P^{(3)}_{\leq4}$ gives an infinite $3$-differential poset with initial rank sequence
\[
  1,3,9,22,50.
\]

\section{The construction for every \texorpdfstring{$r\geq3$}{r >= 3}}\label{sec:general-r}

Fix $r\geq3$ and put $k=\lfloor r/3\rfloor$. Partition the first $3k$ coordinates of $Y^r$ into disjoint triples
\[
  T_j=\{3j+1,3j+2,3j+3\},
  \qquad 0\leq j<k.
\]
For each $T_j$, the elements supported on the coordinates in $T_j$ form an embedded copy of $Y^3$. Perform the replacement $\Rcal\rightsquigarrow\Acal$ inside ranks three and four of each copy.

The affected rank-three vertices belonging to different triples are disjoint. Each replacement preserves all point and pair incidences among its affected vertices. Incidences involving an unaffected vertex are unchanged, and no replacement block mixes two coordinate triples. Thus the identity
\[
  D_4U_3-U_2D_3=rI
\]
remains valid. Each of the $k$ replacements removes one rank-four element, and Proposition~\ref{prop:reflection-extension} extends the resulting partial poset to an infinite $r$-differential poset.

It remains to compute the fourth rank of $Y^r$. Distributing four cells among the $r$ coordinates gives the size patterns
\[
  4,\qquad3+1,\qquad2+2,\qquad2+1+1,\qquad1+1+1+1.
\]
Using $p(1)=1$, $p(2)=2$, $p(3)=3$, and $p(4)=5$, their contributions are
\[
  5r,\qquad3r(r-1),\qquad4\binom r2,\qquad
  2r\binom{r-1}{2},\qquad\binom r4.
\]
Therefore
\begin{align*}
  \card{(Y^r)_4}
  &=5r+3r(r-1)+4\binom r2+2r\binom{r-1}{2}+\binom r4\\
  &=\frac{r(r+1)(r^2+17r+42)}{24}.
\end{align*}
For $r=3$, this is $51$, while the construction gives $50$. This proves Theorem~\ref{thm:main} and disproves \eqref{eq:stanley-bound} in its universally quantified form.

The construction gives
\[
  m_r(4)\leq
  \frac{r(r+1)(r^2+17r+42)}{24}
  -\left\lfloor\frac r3\right\rfloor
  \qquad(r\geq3),
\]
where $m_r(n)$ denotes the minimum possible size of rank $n$ among infinite $r$-differential posets. It neither determines $m_r(4)$ exactly nor addresses $r=1$ or $r=2$.

\section*{Statement on AI-assisted discovery and human verification}
The counterexample presented in this paper was initially generated by the TARS agent system through an autonomous mathematical search and was subsequently examined and independently verified by the human authors.

\end{document}